\documentclass[11pt]{article}
\usepackage{amsmath,amssymb,amsthm}
\usepackage{graphicx}
\usepackage{hyperref}
\usepackage[capitalise]{cleveref}
\usepackage[margin=1in]{geometry}
\usepackage{fancyhdr}
\usepackage{mathtools}
\usepackage{bm}
\usepackage[mathscr]{euscript}
\usepackage{xcolor}
\usepackage{enumitem}
\usepackage{lipsum}
\usepackage{lineno}
\usepackage{tikz}
\usepackage{tikz-cd}
\usepackage{multicol}

\numberwithin{equation}{section}

\newtheorem{theorem}{Theorem}[section]
\newtheorem{corollary}[theorem]{Corollary}
\newtheorem{definition}[theorem]{Definition}

\newtheorem{lemma}[theorem]{Lemma}

\newtheorem{remark}[theorem]{Remark}

\allowdisplaybreaks
\newcommand\hide[1]{}
\newcommand{\TmodXvw}[3]{#1 \backslash\mkern-6mu\backslash (#2)^{ss}_{#1}(#3)}

\begin{document}
	
	\author{
		Somnath Dake \thanks{Chennai Mathematical Institute, Chennai, India \tt{somnath@cmi.ac.in}}
		\and
		Shripad M. Garge \thanks{Indian Institute of Technology Bombay, Powai, Mumbai, India \tt{shripad@math.iitb.ac.in}}
		\and 
		Arpita Nayek \thanks{SRM University, Amaravati, Andhra Pradesh \tt{arpita.n@srmap.edu.in}}
	}
	
	\title{Torus Quotients of Richardson Varieties} 
	\date{}
	
	\maketitle
	
	\begin{abstract}
	 \noindent For $1\le r\le n-1,$ let $G_{r,n}$ denote the Grassmannian parametrizing  $r$-dimensional subspaces of $\mathbb{C}^{n}.$ Let $(r,n)=1.$ In this article we show that the GIT quotients of certain  Richardson varieties in $G_{r,n}$ for the action of a maximal torus in $SL(n,\mathbb{C})$ are the product of projective spaces  with respect to the descent of a suitable line bundle. 
	\end{abstract}

\section{Introduction}\label{section1}
	Let \(T\) be the maximal torus of $SL(n,\mathbb{C})$ consisting of the diagonal matrices in \(SL(n,\mathbb{C})\). There is a natural projective variety structure on $G_{r,n}$ given by the Pl\"ucker embedding $\pi: G_{r,n}\hookrightarrow \mathbb{P}(\wedge^{r}\mathbb{C}^n)$. The restriction of the natural action of $SL(n,\mathbb{ C})$ on $\mathbb{C}^n$ to $T$ induces an action of $T$ on $\wedge^{r}\mathbb{C}^n$ and thus on $\mathbb{P}(\wedge^{r}\mathbb{C}^n),$ moreover, $\pi$ is $T$-equivariant. 
	
	
	For the action of  $T$ on $G_{r,n},$ the GIT quotients have been studied extensively. In \cite{Hausmann1996}, Hausmann and Knutson identified the GIT quotient of the Grassmannian  $G_{2,n}$ by $T$  with the moduli space of polygons in $\mathbb{R}^{3}.$ Also, they showed that GIT quotient of $G_{2,n}$ by $T$ can be realized as the GIT quotient of an $n$-fold product of projective lines by the diagonal action of $PSL(2, \mathbb{C}).$ 
	
	In \cite{Skorobogatov}, Skorobogatov gave a combinatorial description using Hilbert-Mumford criterion, when a point in $G_{r,n}$ is semistable with respect to a suitable $T$-linearized line bundle.  As a corollary he showed that when $(r,n)=1$ semistability is same as the stability. Independently, in \cite{kannan1998}, Kannan showed that semistable locus coincides with stable locus of $G_{r,n}$ if and only if $(r,n)=1$. 
	
	In \cite{kannansardar}, Kannan and Sardar showed that there exists a unique minimal dimensional Schubert variety in $G_{r,n}$ admitting semistable points with respect to a suitable $T$-linearized line bundle. 
	
	One of our motivations behind the study of the GIT quotients of Richardson varieties in $G_{r,n}$ is to produce a variety with nice geometric structures, for example, projective spaces, product of projective spaces or more generally (smooth) toric varieties etc.. 
	
	In this article, we study the GIT quotients of certain Richardson varieties in $G_{r,n}$ for the action of $T$ with respect to the descent of a suitable $T$-linearized line bundle and we show that those are isomorphic to the product of projective spaces as a polarized variety (for more precisely see Theorem~\ref{thm:main}).
	
	The organization of the article is as follows. In \cref{section2} we introduce some notation to define Richardson varieties and recall Standard Monomial Theory on Richardson varieties. We conclude the section by formulating the main result of this article. \cref{section3} and \cref{section4} are devoted to understand the homogeneous coordinate ring of the GIT quotients of the Richardson varieties. In \cref{section3} we develop combinotorial properties of the Young tableaux which arise in the current context and in \cref{section4} we prove our main theorem.
	
\section{Notation and Preliminaries}\label{section2}
We refer to \cite{Hum1}, \cite{Hum2},\cite{LB}, \cite{Mumford}, \cite{Newstead} and \cite{Ses} for preliminaries in algebraic groups, Lie algebras, Standard Monomial Theory and Geometric Invariant Theory.


Let $B~(\supset T)$ be the Borel subgroup of $SL(n,\mathbb{ C})$ consisting of upper triangular matrices. For $1\le i\le n,$ define $\varepsilon_{i}:T\to \mathbb{C}^{\times}$ by  $\varepsilon_{i}(\text{diag}(t_{1},\ldots ,t_{n}))=t_{i}.$ Then  $S:=\{\varepsilon_{i}-\varepsilon_{i+1}| \text{~for all ~} 1\le i\le n-1\}$ forms the set of simple roots of $SL(n, \mathbb{C})$ with respect to $T$ and $B.$ Let $\{\omega_i| 1 \leq i \leq n-1\}$ be the set of fundamental dominant weights corresponding to $S.$ 


Let $I(r,n) := \{\alpha=(\alpha_{1},\alpha_{2},\ldots, \alpha_{r}) | 1\le \alpha_{1}<\alpha_{2}<\cdots <\alpha_{r}\le n\}.$ Let $(e_{1},e_{2},\ldots,e_{n})$ be the standard basis of $\mathbb{C}^n.$  For $\alpha \in I(r,n)$, let $e_{\alpha}=e_{\alpha_1} \wedge e_{\alpha_2} \wedge \cdots \wedge e_{\alpha_r}$. Then $\{e_{\alpha}| \alpha \in I(r,n)\}$ forms a basis of $\wedge^r\mathbb{C}^n.$ Let $\{p_{\alpha}| \alpha\in I(r,n)\}$ be the basis of the dual space $(\wedge^{r}\mathbb{C}^n)^{*},$ which is dual to $\{e_{\alpha}| \alpha \in I(r,n)\}$. These $p_{\alpha}$'s are called Pl\"{u}cker coordinates.

For $1 \leq t \leq n,$ let $V_t$ be the subspace of $\mathbb{C}^n$ spanned by $\{e_1, e_2, \ldots, e_t\}$ and $V^t$ be the subspace of $\mathbb{C}^n$ spanned by $\{e_n, e_{n-1}, \ldots, e_{n-t+1}\}.$ For each $\alpha = (\alpha_{1},\ldots,\alpha_{r}) \in I(r,n)$, the Schubert variety $X_{\alpha}$ associated to $\alpha$ is defined as 
$$X_\alpha=\{U \in G_{r,n} | \text{dim}(U \cap V_{\alpha_t}) \geq t, 1 \leq t \leq r\}.$$
For each $\beta = (\beta_{1},\ldots,\beta_{r}) \in I(r,n),$ the opposite Schubert variety $X^{\beta}$ associated to $\beta$ is defined as 
$$X^\beta=\{U \in G_{r,n} | \text{dim}(U \cap V^{n-\beta_{r-t+1}+1}) \geq t, 1 \leq t \leq r\}.$$
For $\alpha,\beta \in I(r,n),$ the Richardson variety $X^\beta_\alpha$ is defined as $X^\beta \cap X_\alpha.$ Note that $X^\beta_\alpha \neq \emptyset$ if and only if $\beta \leq \alpha$. 

Let $\mathcal{L}(\omega_r)$ be the ample line bundle on $G_{r,n}$ associated to $\omega_r$. For simplicity we also denote the restriction of $\mathcal{L}(\omega_r)$ on $G_{r,n}$ to $X_{\alpha}^{\beta}$ by $\mathcal{L}(\omega_r)$. 

\subsection{Standard Monomial Theory}	
There is a natural partial order on $I(r,n)$, given by \(\alpha \le \beta\) if and only if \(\alpha_{i}\le \beta_i\) for all \(1 \le i \le r\). Let $\tau_1, \tau_2, \ldots, \tau_m \in I(r,n)$. The monomial $p_{\tau_1}p_{\tau_2}\cdots p_{\tau_m} \in H^0(X_{\alpha}^{\beta}, \mathcal{L}(\omega_r)^{\otimes m})$ is said to be standard monomial of degree $m$ if $\beta \leq \tau_1\leq \tau_2\leq \cdots \leq \tau_m \leq \alpha$. 

Let \(M\) and \(N\) be two finite sets. Then a \emph{shuffle permutation \(\sigma\) of \(M\) and \(N\)} is a partition of \(M\cup N\) into two sets \(\sigma(M)\) and \(\sigma(N)\) such that \(|\sigma(M)|=|M|,|\sigma(N)|=|N|\). We can identify all shuffle permutation by \( S_{_{|M|+|N|}}/(S_{_{|M|}}\times S_{_{|N|}}) \).

Let \(\alpha = (\alpha_1, \ldots, \alpha_r),\beta = (\beta_1, \ldots, \beta_r) \in I(r,n)\) such that \(\alpha_{k} > \beta_{k}\) for some \(1 \le k \le r\). Let \(\sigma\) be a shuffle permutation of \(M=\{\alpha_k, \ldots, \alpha_r\}\) and \(N=\{\beta_1, \ldots, \beta_k\}\). Denote by \(\alpha^{\sigma} = (\alpha^{\sigma}_1, \ldots, \alpha^{\sigma}_r)\) and \(\beta^{\sigma} = (\beta^{\sigma}_1, \ldots, \beta^{\sigma}_r)\), the tuples with increasing entries constructed from \(\{\alpha_1, \ldots, \alpha_{k-1}\}\cup \sigma(M)\) and \(\{\beta_{k+1}, \ldots, \beta_r\}\cup \sigma(N)\) respectively. Then the Pl\"ucker relation \(\mathcal{P}(\alpha,\beta,k)\) is
\[
\sum_{\sigma} \pm p_{\alpha^{\sigma}}p_{\beta^{\sigma}}=0 
\]
where sum is over all shuffle permutations of $M$ and $N$. The coefficient is determined by \(sgn(\sigma)\) and the parity required to get tuple \( \alpha^{\sigma} \) and \( \beta^{\sigma} \) with increasing entries. The Grassmannian $G_{r,n} \subseteq \mathbb{P}(\wedge^r \mathbb{C}^n)$ is precisely the zero set of these Pl\"ucker relations. Let \(I(X_{\alpha}^{\beta})\) be the ideal of \(X_{\alpha}^{\beta}\) in \(\mathbb{C}[\wedge^{r}\mathbb{C}^{n}]\). We have following theorem for Richardson variety \(X^{\beta}_{\alpha}\). 
\begin{theorem}
	\begin{enumerate}[label=\arabic*.,ref=\arabic*.,wide,itemsep=0pt, labelwidth=!, labelindent=0pt]
		\item The standard monomials of degree $m$ on $X^{\beta}_{\alpha}$ form a basis of $H^0(X^{\beta}_{\alpha},\mathcal{L}(m\omega_r)).$
		\item The ideal \(I(X^{\beta}_{\alpha})\) in \(\mathbb{C}[\wedge^{r}\mathbb{C}^{n}]\) is generated by Pl\"ucker relations and the set \(\{p_{\tau} | \tau \not\le \alpha \text{ or }\tau \not\ge \beta\}\).
		
\end{enumerate}	\end{theorem}



%

Let \(\lambda = m\omega_{r}\) be a dominant weight. A Young diagram associated to \(\lambda\) (which is also denoted by $\lambda$) contains \(r\) rows with $m$ boxes in each row. A Young tableau \(\Gamma\) of shape \(\lambda\) is a filling of Young diagram \(\lambda\) with the integers in \(\{1,\ldots,n\}\). We denote by \(col_{i}(\Gamma)\) the \(i\)-th  column of \(\Gamma\). Let \(\Gamma_1\) and \(\Gamma_2\) be two Young tableaux of shape \(d_{1}\omega_r\) and \(d_{2}\omega_r\) respectively. Then we define the product tableau \(\Gamma_1\Gamma_2\) of shape \((d_{1}+d_{2})\omega_r\) such that for \(1 \le i \le d_{1}\), \(col_{i}(\Gamma_1\Gamma_2) = col_{i}(\Gamma_1)\) and for \(d_{1}+1 \le i \le d_{1}+d_{2}\), \(col_{i}(\Gamma_{1}\Gamma_2) = col_{i-d_{1}}(\Gamma_2)\). A Young tableau is called column standard if the entries along any column are strictly increasing from top to bottom and called row semistandard if the entries along any row are non-decreasing from left to right. A Young tableau is called semistandard if it is column standard and row semistandard. 

For any two integers $i < j,$ we denote the integers in $\{i, i+1,\ldots, j\}$ by $[i, j].$ Then the boxes in Young diagram \(\lambda\) are indexed by the set \([1,r]\times[1,m]\).  We denote the entry in the box \((i,j)\) by \(\Gamma(i,j)\) .

For any column standard tableau  \(\Gamma\) of shape \(\lambda\), we associate a monomial $f_{\Gamma}$ by taking product of Pl\"ucker coordinates indexed by columns of \(\Gamma\). Conversely, to any monomial \(p_{\tau_1}\cdots p_{\tau_m}\), we associate the column standard tableau \(\Gamma\) of shape \(m\omega_r\) such that \(col_{i}(\Gamma) = \tau_i\). Note that for $\alpha, \beta\in I(r,n),$ $f_{\Gamma}$ is a non-zero standard monomial on $X_{\alpha}^{\beta}$ if and only if $\Gamma$ is semistandard and the last column of $\Gamma$ is less than or equal to $\alpha$ and the first column is greater than or equal to $\beta$. We use \(\Gamma\) and \(f_{\Gamma}\) interchangeably using the identification above.

For \(d \in \mathbb{Z}_{+}\), let \(\lambda_{d} = nd\omega_{r}\). Let $\Gamma$ be a Young tableau of shape $\lambda_d$. For a positive integer $1\leq i\leq n$, we denote by $wt_{i}(\Gamma)$, the number of boxes of $\Gamma$ containing the integer $i$. The weight of $\Gamma$ is defined as $wt(\Gamma):=wt_{1}(\Gamma)\varepsilon_1+\cdots+wt_{n}(\Gamma)\varepsilon_{n}$  (see \cite[Section 2, p.336]{LB}). We conclude this section with the key result that a monomial $f_{\Gamma}\in H^0(G_{r,n},\mathcal{L}(\lambda_d))$ is $T$-invariant if and only if \(wt_i(\Gamma) = wt_j(\Gamma)\) for all $1 \leq i,j \leq n$. It follows that $wt_i(\Gamma)=dr$ for all $1 \leq i \leq n$ $($See \cite[Lemma 3.1, p.4]{NP}$)$.

%
%


\subsection{Main result}

Recall that by \cite[Corollary 1.9, p.~85]{kannansardar}, there exists a unique element $w \in I(r,n)$ of minimal length such that $w(\omega_r) \leq 0.$ Thus, by \cite[Lemma 2.1, p.470]{kannanpattanayak} $X(w)^{ss}_T(\mathcal{L}(\omega_r))$ is non-empty. It follows that $w =(a_1, a_2, \ldots, a_r),$ where $a_i$ is the smallest positive integer such that $ra_i \geq ni$ (see \cite[Proposition 2.2, p.895]{Bakshi}).  Thus for $1 \leq i \leq r$, \(a_i = \lceil\frac{ni}{r}\rceil\). 

 We set $a_0:=0.$ For \(0 \le i \le r,\) let \(c_i = ra_i-ni\). Note that $c_0=0$. Since $a_r=n$ we have $c_r=0$.  Now we have the following lemma.  
\begin{lemma}\label{lemma:property_c}
	For \(1 \leq i \leq r-1\) we have $0 < c_i <r$.
\end{lemma} 

\begin{proof}
By the definition of $c_i$ we have $c_i \geq 0$. Since $r$ and $n$ are coprime we have $c_i \neq 0$. Thus we have $c_i > 0$. If $c_i \geq r$ then $r(a_i-1) \geq ni$ which is a contradiction to the fact that $a_i$ is the smallest positive integer such that $ra_i \geq ni$. Thus we have \(c_i < r\).
\end{proof}

Let $v \in I(r,n)$ be the unique maximal element such that $v(\omega_r)\geq 0$. Then we have $w=cv,$ for some Coxeter element $c$ such that $l(w)=l(v)+(n-1)$. Further, we have $v=(1,a_1,a_2,\ldots,a_{r-1}).$ It is natural to study the GIT quotient of Richardson variety $X^{\beta}_w$ for some $\beta \leq v$.   

For a sequence of integers $l_i$ $(1 \leq i \leq r-1)$ such that $a_{i-1}+2 \leq l_i \leq a_i$, let $\underline{l}=(1,l_1, l_2, \ldots, l_{r-1})$ and 
$v_{\underline{l}} = (1, l_{1}, l_{2}, \ldots, l_{r-1}).$ Note that $v_{\underline{l}} \leq v.$ By \cite[Lemma 2.1, p.470]{kannanpattanayak} $ (X^{v_{\underline{l}}}_{w})^{ss}_T(\mathcal{L}(\omega_r))$ is non-empty. We set $l_0:=1$ and $l_r:=a_r+1$. Note that we have following inequalities which we use later in throughout the document. 
\begin{equation}\label{sequence} 
	a_{0}< a_0+2 \leq l_1 \leq a_1< \cdots \leq a_{i-1} < a_{i-1}+2 \leq l_i \leq a_i< \cdots \leq a_{r-1} < a_{r-1}+2 \leq l_r 
	\end{equation}

We recall that by \cite[Theorem 3.10, p.764]{Kumar} the line bundle $\mathcal{L}(n\omega_r)$ descends to a line bundle on the GIT quotient $\TmodXvw{T}{G_{r,n}}{\mathcal{L}(n\omega_r)}$. In this article, we  study the torus quotients of the Richardson varieties $X^{v_{\underline{l}}}_{w}$ with respect to the descent of the line bundle $\mathcal{L}(n\omega_r)$. Note that $\TmodXvw{T}{X^{v_{\underline{l}}}_{w}}{\mathcal{L}(n\omega_r)} = Proj(R)$ where \(R = \oplus_{d \in \mathbb{Z}_{\ge 0}} R_d\) and $R_d= H^0(X^{v_{\underline{l}}}_{w}, \mathcal{L}(n\omega_r)^{\otimes d})^T$.
The main theorem of this article is the following:
\begin{theorem}\label{thm:main}
	The GIT quotient $\TmodXvw{T}{X^{v_{\underline{l}}}_{w}}{\mathcal{L}(n\omega_r)}$ is isomorphic to ${\mathbb P}^{a_1-l_1} \times {\mathbb P}^{a_2-l_2} \times \cdots \times {\mathbb P}^{a_{r-1}-l_{r-1}}$ and it is embedded via the very ample line bundle $\mathcal{O}_{\mathbb{P}^{a_1-l_1}}(c_1) \boxtimes \mathcal{O}_{\mathbb{P}^{a_2-l_2}}(c_2) \boxtimes \cdots \boxtimes \mathcal{O}_{\mathbb{P}^{a_{r-1}-l_{r-1}}}(c_{r-1})$. 
\end{theorem}

This result generalizes the earlier work of Kannan and Nayek (see \cite{kannannayek}) using different combinatorial approach.   

	\section{Structure of \(T\)-Invariant Tableaux}\label{section3}
	
	We fix an integer \(d \in \mathbb{Z}_{\geq 1}\) throughout this section. Let \(\mathtt{ST}(\lambda_d)\) denotes the set of all $T$-invariant semistandard Young tableaux of shape \(\lambda_d\) associated to non-zero standard monomials in $R_d$. Set \[\mathtt{ST} := \bigcup_{d\in \mathbb{Z}_{\ge 1}}\mathtt{ST}(\lambda_d).\]
	Note that a tableau \(\Gamma\) is in \(\mathtt{ST}(\lambda_d)\) if and only if 
	\begin{enumerate}[label=(\alph*), ref=\text{Condition}~(\alph*), itemsep=0pt]
		\item\label{Tinvariance:condition1} $\Gamma$ is semistandard.
		\item\label{Tinvariance:condition2} $\Gamma$ contains each integer \(1 \leq k \leq n\) with multiplicity \(rd\).
		\item\label{Tinvariance:condition3} $\Gamma(i,j) \in \{l_{i-1},\ldots,a_{i}\}$ for \(1 \le i \le r\) and  for \(1 \le j \le nd\).
	\end{enumerate} 
	In this section we study the combinatorial properties of tableaux which are going to arise while studying the algebra \(R\). In \cref{T_inv_eq_condition} we give an equivalent condition for a Young tableau of shape $\lambda_d$ to be in $\mathtt{ST}(\lambda_d)$. In \cref{lattice_bij} we define a set of integer points (\cref{def:polytope}) and construct a bijection between the set of integer points and \(\mathtt{ST}(\lambda_d)\). In \cref{splitting} we associate two tableaux \(\Gamma^{(1)},\Gamma^{(2)}\) to a tableau \(\Gamma \in \mathtt{ST}(\lambda_d)\) and study their combinatorial properties which we need to understand the Pl\"{u}cker relations.
	
	\subsection{Equivalent condition for \(T\)-invariance}\label{T_inv_eq_condition}
	
	Let \(\Gamma\) be a Young tableau of shape $\lambda_d$. The main purpose of this section is to state \cref{lemma:Tinv-eq-condition} in which we give equivalent condition for $\Gamma \in \mathtt{ST}(\lambda_d)$. The proof of \cref{lemma:Tinv-eq-condition} is bit technical and involves manipulations of entries in the tableau \(\Gamma\). To make article readable we have moved the proof in \cref{proof:Tinv_eq_condition}.
	\begin{lemma} Let \(\Gamma \in \mathtt{ST}(\lambda_d)\). For $1 \leq i \leq r$, the boxes in \(\Gamma\) which contain integers in \(\{1,\ldots,a_i\}\) are \([1,i] \times [1,dn] \cup \{i+1\}\times [1,dc_i].\) \label{lemma:boxes} 
	\end{lemma}
	\begin{proof} 
	For $i=r$ the claim is true.
	
	Let $i=r-1$. By \labelcref{Tinvariance:condition3} \(\Gamma(r-1,dn) \leq a_{r-1}\). Since $\Gamma(r-1,dn)$ is the largest entry in first \(r-1\) rows, all boxes in first $r-1$ rows contain integers in \(\{1,\ldots,a_{r-1}\}\). From \labelcref{Tinvariance:condition2} the number of boxes in $\Gamma$ containing integers in $\{1,\ldots,a_{r-1}\}$ are $dra_{r-1}$. Hence number of boxes in row $r$ which contain integers in \(\{1,\ldots,a_{r-1}\}\) is \(dra_{r-1} - dn(r-1) (=dc_{r-1})\). Therefore, the boxes in \(\Gamma\) which contain integers in \(\{1,\ldots,a_{r-1}\}\) are \([1,r-1] \times [1,dn] \cup \{r\}\times [1,dc_{r-1}]\).
	
	Let $1 \leq i \leq r-2$. The proof of this case is similar to the previous case with some minor changes. Since $\Gamma(i+2,1) \geq l_{i+1} > a_{i},$ the integers in \(\{1,\ldots,a_i\}\) cannot appear in rows $\{i+2, \ldots, r\}.$ Further, since \(\Gamma(i,dn) \leq a_{i}\) (\labelcref{Tinvariance:condition3}) is the largest entry in first \(i\) rows, all boxes in first $i$ rows contain integers in \(\{1,\ldots,a_i\}\). From \labelcref{Tinvariance:condition2} the number of boxes in $\Gamma$ containing integers in $\{1,\ldots,a_{i}\}$ are $dra_{i}$. Hence number of boxes in row $i+1$ which contain integers in \(\{1,\ldots,a_i\}\) is \(dra_{i} - dni (=dc_i)\). Therefore, the boxes in \(\Gamma\) which contain integers in \(\{1,\ldots,a_i\}\) are \([1,i] \times [1,dn] \cup \{i+1\}\times [1,dc_{i}]\).
	\end{proof}
  Note that the boxes in Young diagram $\lambda_d$ are $[1,r] \times [1,dn]$. We partition these boxes and integers in \(\{1,2,\ldots,n\}\) as follows:
	\[
	[1,r]\times[1,dn] = \bigsqcup_{i=1}^{r} B_i \text{ and }
	\{1,2,\ldots,n\} = \bigsqcup_{i=1}^{r} C_i		
	 	\]
	where
	\[
	\begin{array}{ccll}
	&B_i = B_{i,1} \sqcup B_{i,2}, &B_{i,1} =\{i\}\times[dc_{i-1}+1,dn],&B_{i,2}=\{i+1\} \times [1,dc_{i}],\\
		&C_i = C_{i,1} \sqcup C_{i,2}, &C_{i,1} = \{a_{i-1}+1,\ldots,l_i-1\}, &C_{i,2}=\{l_i,\ldots,a_i\}.
			\end{array}
	\]
We define \(C :=\sqcup_{i=1}^{r}C_{i,2}\).  Since $c_r=0,$ we have $B_{r,2}$ is empty and since $l_r=a_r+1,$ we have $C_{r,2}$ is empty.

Note that definition of \(B_{i}\) depends on \(d\) but we purposely skip to include this in the notation because it has little or no advantage and skipping it makes notation easier. 

\begin{corollary}
	Let \(\Gamma \in \mathtt{ST}(\lambda_d)\). For \(1 \le i \le r\), the boxes in \(\Gamma\) which contain integers in \(C_{i}\) are \(B_{i}\). Moreover for \(1 \le i \le r-1\), the boxes in \(\Gamma\) which contain integers in \(C_{i,2}\) are \(B_{i,2}\).
	\label{cor:BicontainsCi}
\end{corollary} 
\begin{proof}
	Let \(1 \le i \le r\). Observe that \(\{1,\ldots,a_i\} = \sqcup_{k=1}^{i} C_{k}\) and \([1,i] \times [1,dn] \cup \{i+1\}\times [1,dc_i] = \sqcup_{k=1}^{i} B_{k}\). From \cref{lemma:boxes} and using induction on $i$, it follows that the boxes in \(\Gamma\) which contains integers from \(C_{i}\) are \(B_{i}\).
	
	Let \(1 \le i \le r-1\). Since \(\Gamma(i+1,j) \in \{l_i,\ldots, a_{i+1}\}\) for any $j$, the boxes in \(B_{i,2}\) contain integers in \(C_{i} \cap \{l_{i},\ldots,a_{i+1}\} = \{l_{i},\ldots,a_{i}\} = C_{i,2}\) (see \ref{sequence}).
\end{proof}

\begin{corollary}\label{cor:identicalBi2}
	Let $\Gamma_1, \Gamma_2 \in \mathtt{ST}(\lambda_d)$. For $1\leq i \leq r$, if $\Gamma_1$ and $\Gamma_2$ contain identical entries in $B_{i,2}$ then $\Gamma_1=\Gamma_2.$ 
\end{corollary}
\begin{proof}
	Let \(1 \le i \le r\). From \cref{cor:BicontainsCi} the subtableaux of $\Gamma_1$ and $\Gamma_2$ containing boxes $B_i$ are filled with each integer in $C_i$ with multiplicity $rd$. Since $\Gamma_1$ and $\Gamma_2$ contain identical entries in $B_{i,2}$, the boxes $B_{i,1}$ in $\Gamma_1$ and $\Gamma_2$ contain identical entries. 
	\end{proof}
 
	\begin{lemma}\label{lemma:Tinv-eq-condition}
		Let \(\Gamma\) be a Young tableau of shape $\lambda_d$. Then $\Gamma \in \mathtt{ST}(\lambda_d)$ if and only if for all \(1 \leq i \leq r\) we have
		\begin{enumerate}[label=$(\arabic*)$,itemsep=0pt]
			\item Subtableau of \(\Gamma\) containing boxes \(B_i\) is row semistandard skew tableau.
			\item Boxes \(B_i\) contain integers in \(C_i\) each with multiplicity \(rd\).
			\item \(\Gamma(i,1) \ge l_{i-1}\).
		\end{enumerate}
	\end{lemma}	
	\begin{proof}
		See \cref{proof:Tinv_eq_condition}
	\end{proof}
	\begin{corollary}\label{cor:colrd}
		Let	\(\Gamma \in \mathtt{ST}(\lambda_{d})\). Then \(\Gamma(i,rd) = a_{i-1}+1\) for all $1 \leq i \leq r$.
	\end{corollary}
	\begin{proof}
		Follows from \labelcref{step:colrd} of proof of \cref{lemma:Tinv-eq-condition}.	
	\end{proof}

\subsection{Lattice points and \(T\)-invariants}\label{lattice_bij}
Recall that $C=\sqcup_{i=1}^{r} C_{i,2}$. Let $\lambda=m\omega_r$. In this section we associate an integer point in $\mathbb{Z}^{|C|}$ to a tableau $\Gamma$ of shape $\lambda$ (need not be semistandard) and we define a set of integer points $\mathtt{P}_d$. The main takeaway of this section is \cref{lemma:bij} in which we prove that there is a bijection between $\mathtt{P}_d$ and $\mathtt{ST}(\lambda_d)$. 
	\begin{definition}
		Let \(\Gamma\) be a Young tableau of shape \(\lambda\) (need not be semistandard). For any $j \in C,$ there exists unique $1 \leq i \leq r-1$ such that $j\in C_{i,2},$ we define \( \mathtt{z}_j(\Gamma) \) to be the number of times \( j \) appears in row \({i+1}\) of \(\Gamma \). We set \( \mathtt{z}(\Gamma) := \big(\mathtt{z}_j(\Gamma)\big)_{j \in C} \in \mathbb{Z}_{\ge 0}^{|C|} \). 
		Denote the restriction of \(\mathtt{z}\) to \(\mathtt{ST}(\lambda_d)\) by \(\bar{\mathtt{z}}_d\).
		\label{def:z}
	\end{definition}
	
	\begin{remark}
		Let \(\Gamma_1\) and \(\Gamma_2\) be two Young tableaux of shape $m_1\omega_r$ and $m_2\omega_r$ respectively (need not be semistandard). Then \(\mathtt{z}(\Gamma_1\Gamma_2)=\mathtt{z}(\Gamma_1)+\mathtt{z}(\Gamma_2)\).
		\label{def:z:adds}
	\end{remark}

	\begin{definition}\label{def:polytope}
	We define a set of integer points $$\mathtt{P}_d:=\Big\{\bar{z}=(z_{j})_{_{j \in C}} \big{|} \text{ for } 1 \leq i \leq r-1, \sum_{j \in C_{i,2}} z_j = dc_i\Big\} \subseteq \mathbb{Z}^{|C|}_{\ge 0}.$$
	\end{definition}
 
 \begin{lemma}
 	If \(\Gamma \in \mathtt{ST}(\lambda_d)\) then \(\mathtt{z}(\Gamma) \in \mathtt{P}_{d}\).
 	\label{lemma:zinPd}
 \end{lemma}
 \begin{proof}
 	Let \(1 \le i \le r-1\). We prove that \(\sum_{j \in C_{i,2}} \mathtt{z}_{j}(\Gamma)= dc_i\). Observe that boxes in row \(i+1\) are \(B_{i,2} \sqcup B_{i+1,1}\). From \cref{cor:BicontainsCi} we have \(B_{i,2}\) contains integers from \(C_{i,2}\) and boxes \(B_{i+1,1}\) contains integers from \(C_{i+1}\). Hence \(C_{i}\) and \(C_{i+1}\) are disjoint implies \(\sum_{j \in C_{i,2}} \mathtt{z}_{j}(\Gamma) = |B_{i,2}| = dc_{i}\). Hence \(\mathtt{z}(\Gamma)  \in \mathtt{P}_{d} \).
\end{proof}
\begin{lemma} 
		The map \(\bar{\mathtt{z}}_d\) is a bijection between \(\mathtt{P}_d\) and \(\mathtt{ST}(\lambda_d)\).
		\label{lemma:bij}
\end{lemma}

\begin{proof}
Let \(\bar{z} \in \mathtt{P}_d \). We construct a semistandard tableau \( \Gamma_{\bar{z}}\) and show that \(\Gamma_{\bar{z}} \in \mathtt{ST}(\lambda_d)\). Let $1 \leq i \leq r$.
We start with Young diagram \(\lambda_d\) and fill the boxes \(B_{i} ( = B_{i,1}\sqcup B_{i,2})\) with the entries from \(C_{i}( = C_{i,1}\sqcup C_{i,2})\) each with multiplicity \(rd\) in non-decreasing order in each row (This is possible because for \(1 \leq i \leq r\) we have $|B_i| = rd(a_i-a_{i-1})$ and \(|C_{i}| = a_{i}-a_{i-1}\)). We do this as follows.
\begin{enumerate}
	\item Fill the boxes in \(B_{i,2}\) with the integers \(j \in C_{i,2}\) each with multiplicity \(z_j\) in non-decreasing order. This is possible because $\sum_{j \in C_{i,2}} z_j =dc_i$ and number of boxes in $B_{i,2}$ is also $dc_i$. 
	\item Fill the boxes \(B_{i,1}\) with the integers \(j \in C_{i,1} \) each with multiplicity \( rd \) and the integers \( j \in C_{i,2} \) each with multiplicity \( rd-z_j \) in non-decreasing order.
\end{enumerate}
Then \(\Gamma_{\bar{z}}\) has the following properties.
For all \(1 \leq i \leq r\) 
\begin{enumerate}[label=(\arabic*),ref=\arabic*.,itemsep=0pt]
	\item Subtableau of \(\Gamma_{\bar{z}}\) containing boxes \(B_i\) is row semistandard skew tableau.
	\item Boxes \(B_i\) contain integers from \(C_i\) each with multiplicity \(rd\).
	\item \(\Gamma_{\bar{z}}(i,1) \ge l_{i-1}\).
\end{enumerate}
Hence from \cref{lemma:Tinv-eq-condition} we have \(\Gamma_{\bar{z}} \in \mathtt{ST}(\lambda_d)\).	

Now we prove that the map $\bar{z} \mapsto \Gamma_{\bar{z}}$ is a bijection and $\bar{\mathtt{z}}_d$ is the inverse. If \(\bar{z},\bar{z}' \in \mathtt{P}_d\) and \(\bar{z} \ne \bar{z}'\),\hide{Then $z_j \neq z'_j$ for some $j \in C_{i,2}$ for some $i \in [1,r-1]$,} then from construction \(\Gamma_{\bar{z}} \ne \Gamma_{\bar{z}'}\) follows. This proves injectivity.  
For surjectivity, let \(\Gamma' \in \mathtt{ST}(\lambda_d)\). Then \(\mathtt{z}(\Gamma') \in \mathtt{P}_d\) (From \cref{lemma:zinPd}). We claim that \(\Gamma_{\mathtt{z}(\Gamma')} = \Gamma'\). From construction of \(\Gamma_{\mathtt{z}(\Gamma')}\) we have, for all \(1 \leq i \leq r\) the boxes \(B_{i,2}\) in \(\Gamma_{\mathtt{z}(\Gamma')}\) contain integers \(j \in C_{i,2}\) each with multiplicity \(\mathtt{z}_j(\Gamma')\) in non-decreasing order. 
Note that by the definition of \(\mathtt{z}_j(\Gamma')\) and the assumption that \(\Gamma' \in \mathtt{ST}(\lambda_d)\), we have for all \(1 \leq i \leq r\) boxes \(B_{i,2}\) in \(\Gamma'\) contain integers \(j \in C_{i,2}\) each with multiplicity \(\mathtt{z}_j(\Gamma')\) in non-decreasing order. Hence the entries in boxes \(B_{i,2}\) in $\Gamma'$ and $\Gamma_{\mathtt{z}(\Gamma')}$ are identical. Therefore from \cref{cor:identicalBi2} we have \(\Gamma'=\Gamma_{\mathtt{z}(\Gamma')}\).
\end{proof}	
	
\subsection{Splitting \(T\) invariant tableaux}\label{splitting}
	In this section we split a Young tableau in $\mathtt{ST}(\lambda_d)$ into two subtableaux (\cref{def:split}) and we study image of $\mathtt{z}$ on each subtableau.
	
	Let \(\Gamma\) be a tableau of shape \(m\omega_r\) (need not be semistandard). We denote the following two predicates by P$1$ and P$2$ .
			 
\begin{enumerate}[label=P\arabic*$(\Gamma)$:, ref=P\arabic*,leftmargin=50pt,itemsep=0pt]
			\item \label{property:P1} For \( 1 \leq i \leq r \) and $1 \leq j \leq m$, \(\Gamma(i,j) \in \{l_{i-1},\ldots,a_{i-1}+1\}\).
			\item \label{property:P2} For \( 1 \leq i \leq r \) and $1 \leq j \leq m$, \(\Gamma(i,j) \in \{a_{i-1}+1,\ldots,a_i\}\).
		\end{enumerate}
An arbitrary $\Gamma$ may satisfy \ref{property:P1} or \ref{property:P2} or both or none.
	\begin{remark} We have following remarks.
		\begin{enumerate}[label={\thetheorem\alph*}., ref={\thetheorem\alph*}, leftmargin=40pt]
			\item \label{prop-remark:mult}
				If \(\Gamma\) and \(\Gamma'\) are two tableaux of shape $m_1\omega_r$ and $m_2\omega_r$ respectively satisfying \ref{property:P1}, then \(\Gamma\Gamma'\) also satisfies \ref{property:P1}. Similarly, if \(\Gamma\) and \(\Gamma'\) both satisfy \ref{property:P2}, then product \(\Gamma\Gamma'\) also satisfies \ref{property:P2}.
			\item \label{prop-remark:sst}
				Let \(\Gamma_1\) be a semistandard tableau of shape $m_1\omega_r$ and \(\Gamma_2\) be a semistandard tableau of shape $m_2\omega_r$. If $\Gamma_1$ and $\Gamma_2$ satisfy~\ref{property:P1} and~\ref{property:P2} respectively, then \(\Gamma_1\Gamma_2\) is semistandard. This follows from $\Gamma_1(i,m_1) \leq a_{i-1}+1$ and \(\Gamma_2(i,1) \geq a_{i-1}+1\) for all $1 \leq i \leq r$.
		\end{enumerate}
	\end{remark}
	\begin{lemma}
		Let \(\Gamma\) be a tableau of shape $m\omega_r$ which satisfies \ref{property:P2}. Then \(\mathtt{z}(\Gamma) = 0\).
		\label{lemmaP2:z0}
	\end{lemma}
	\begin{proof}
		For \(1 \leq i \leq r-1\) and \(j \in C_{i,2}\), \(\mathtt{z}_j(\Gamma)\) is number of times \(j\) appear in row \(i+1\). Since \(\Gamma\) satisfies \ref{property:P2} we have \(\Gamma(i+1,j) \in \{a_{i}+1, \ldots, a_{i+1}\}\). The claim follows from \(C_{i,2} \cap \{a_{i}+1, \ldots, a_{i+1}\} = \emptyset\).
	\end{proof}
	\begin{definition}
		Let \({\Gamma} \in \mathtt{ST}(\lambda_d)\). We associate two tableaux \(\Gamma^{(1)}\) and \(\Gamma^{(2)}\) to \(\Gamma\) as follows:
		Define \(\Gamma^{(1)}\) by taking columns \(1,\ldots,rd\) of $\Gamma$ and \(\Gamma^{(2)}\) by taking columns \(rd+1,\ldots, nd	\) of \(\Gamma\). We have \(\Gamma = \Gamma^{(1)}\Gamma^{(2)}\) and the shape of \(\Gamma^{(1)}\) and \(\Gamma^{(2)}\) are \(rd\omega_r\) and \((n-r)d\omega_r\) respectively.
		\label{def:split}
	\end{definition}
	\begin{remark}
		Let \({\Gamma} \in \mathtt{ST}(\lambda_d)\). Then \(\Gamma^{(1)}\) satisfies \ref{property:P1} and \(\Gamma^{(2)}\) satisfies \ref{property:P2}. Hence \(\mathtt{z}(\Gamma^{(2)})= 0\) and \(\mathtt{z}(\Gamma) = \mathtt{z}(\Gamma^{(1)})\).
		\label{def:split:P1P2z}
	\end{remark}
	\begin{proof}		
	For $1 \leq i \leq r$ and $1 \leq j \leq rd$ we have $\Gamma^{(1)}(i,j) \in \{l_{i-1}, \ldots, a_{i-1}+1\}$ follows from $\Gamma^{(1)}$ is semistandard, $\Gamma^{(1)}(i,1)=\Gamma(i,1) \geq l_{i-1}$ and $\Gamma^{(1)}(i,rd)=\Gamma(i,rd) = a_{i-1}+1$ (see \cref{cor:colrd}). Hence \(\Gamma^{(1)}\) satisfies \ref{property:P1}. Similarly \(\Gamma^{(2)}\) satisfies \ref{property:P2} follows from $\Gamma^{(2)}$ is semistandard, $\Gamma^{(2)}(i,1) = \Gamma(i,rd+1) \geq \Gamma(i,rd) = a_{i-1}+1$ (see \cref{cor:colrd}) and $\Gamma^{(2)}(i,(n-r)d)=\Gamma(i,nd) = a_{i}$.
	
By \cref{lemmaP2:z0} \(\mathtt{z}(\Gamma^{(2)})= 0\). Hence by \cref{def:z:adds} \(\mathtt{z}(\Gamma) = \mathtt{z}(\Gamma^{(1)})\) .
	\end{proof}

\section{Ring of Invariants}\label{section4}

	In this section we prove the main theorem. We do this by proving algebra \(R\) is isomorphic to the homogeneous coordinate ring of product of projective spaces. In  \cref{sec:pluckrel} we prove that the straightening relations on $X^{v_{\underline{l}}}_w$ are binomials. Next in \cref{sec:proofmain} we prove the main theorem.
\subsection{Pl\"ucker relations}
\label{sec:pluckrel}
	We now prove that straightening relations on $X^{v_{\underline{l}}}_w$ are binomials.
 Recall that for $\alpha=(\alpha_1, \ldots, \alpha_r)\in I(r,n)$ the Pl\"ucker coordinate $p_{\alpha}$ is non-zero on $X^{v_{\underline{l}}}_w$ if and only if \(\alpha_i \in [l_{i-1},a_i] = [l_{i-1},a_{i-1}+1] \cup [a_{i-1}+1,a_i]\) for all $1 \leq i \leq r$. \hide{We note following two observations which will be used in the proofs of Lemma{}. For \( 1 \le i,j \le r \),
	\begin{enumerate}
		\item[(i)] \textcolor{red}{if} \(j<i\) then \([l_{j-1},a_{j-1}+1] \cap [l_{i-1},a_i] = \emptyset\)
		\item[(ii)] \textcolor{red}{if} \(i<j\) then \([a_{j-1}+1,a_j] \cap [l_{i-1},a_i] = \emptyset\).
	\end{enumerate}}
	\begin{lemma}
		Let \(\Gamma\) be a column standard Young tableau of shape $2\omega_r$ which satisfies \ref{property:P1} such that \( col_1(\Gamma)=\alpha = (\alpha_1, \ldots, \alpha_r) \in I(r,n)\) and \(col_2(\Gamma)=\beta = (\beta_1, \ldots, \beta_r) \in I(r,n)\). Let \(1 \leq k \leq r\) be a positive integer such that \(\alpha_1 \le \beta_1,\ldots,\alpha_{k-1} \le \beta_{k-1}, \alpha_{k} > \beta_{k}\). Then on the Richardson variety \(X^{v_{\underline{l}}}_{w}\), the Pl\"ucker relation $\mathcal{P}(\alpha, \beta, k)$ is 
		\[
		p_{\alpha}p_{\beta} - p_{\alpha^{\sigma}}p_{\beta^{\sigma}}=0
		\]
		where \( \sigma\) is the transposition \((\alpha_{k},\beta_{k}) \).
		
		Moreover, let $\sigma(\Gamma)$ be the column standard Young tableau of shape $2\omega_r$ with \(col_1(\sigma(\Gamma))=\alpha^{\sigma}\) and \(col_2(\sigma(\Gamma))=\beta^{\sigma}\). Then $f_{\Gamma}=f_{\sigma(\Gamma)}$. Also, \(\sigma(\Gamma)\) satisfies \ref{property:P1} and \(\mathtt{z}(\Gamma) = \mathtt{z}(\sigma(\Gamma))\).
		
		\label{lemma:firstpartrel}
	\end{lemma}
	\begin{proof}
		We will prove that in the Pl\"ucker relation \(\mathcal{P}(\alpha,\beta,k) \), the monomial \( p_{\alpha^{\sigma}}p_{\beta^{\sigma}}\) is non-zero on the Richardson variety $X^{v_{\underline{l}}}_w$ if and only if \( \sigma =id \) or \(\sigma = (\alpha_{k},\beta_{k})\). Observe that if $\sigma$ is other than $id$ and the transposition $(\alpha_k, \beta_k)$ then either (a) \(\beta_j\in \alpha^{\sigma}\) for some \(1 \le j \le k-1\) or (b) \(\alpha_{k},\beta_{k} \in \alpha^{\sigma}\). 
	
		(a) Let \(\sigma\) be such that \(\beta_j\in \alpha^{\sigma}\) for some \(1 \le j \le k-1\). We want to understand the position of \(\beta_j\) in \(\alpha^{\sigma}\). Let \(1 \le i \le r\) be such that \(\alpha^{\sigma}_i = \beta_j\).  Note that \( \{\alpha_1, \ldots, \alpha_{j}\} \) all are in \( \alpha^{\sigma} \) and smaller than  \( \beta_j\), hence we have \(i>j\). Since \(\beta_{j} \in \{l_{j-1}, \ldots, a_{j-1}+1\}\) and \(i>j\) implies \(\{l_{j-1}, \ldots, a_{j-1}+1\} \cap \{l_{i-1}, \ldots, a_i\} = \emptyset\), we have that the Pl\"ucker coordinate \(p_{\alpha^{\sigma}}\) is zero on \(X^{v_{\underline{l}}}_w\). 
		
		(b) Let \(\sigma\) be such that \(\alpha_{k},\beta_{k} \in \alpha^{\sigma}\). We want to understand the position of \(\alpha_k\) in \(\alpha^{\sigma}\). Let \(1 \le i \le r\) be such that \(\alpha^{\sigma}_i = \alpha_k\). Observe that \(\alpha^{\sigma}\) contains \(\{\alpha_1, \ldots, \alpha_{k-1},\beta_{k}\}\) and all are strictly smaller than \(\alpha_k\). Hence we have \(i > k\). Since \(\alpha_k \in \{l_{k-1}, \ldots, a_{k-1}+1\}\) and $i > k$ implies \(\{l_{k-1}, \ldots, a_{k-1}+1\} \cap \{l_{i-1}, \ldots, a_i\} = \emptyset\), we have that the Pl\"ucker coordinate \(p_{\alpha^{\sigma}}\) is zero on \(X^{v_{\underline{l}}}_w\).
	
	Note that $\alpha^{\sigma}=(\alpha_1 < \cdots < \alpha_{k-1} < \beta_k< \alpha_{k+1}< \cdots< \alpha_r)$ follows from $\alpha_i, \beta_i \in [l_{i-1},a_{i-1}+1]$ (see \ref{property:P1}). Similarly, we have $\beta^{\sigma}=(\beta_1< \cdots< \beta_{k-1} < \alpha_k < \beta_{k+1}< \cdots< \beta_r)$. It follows that \(\sigma(\Gamma)\) satisfies \ref{property:P1}.  
	
	Since $\mathtt{z}(\Gamma)$ is invariant under permuting entries inside any row, we have \(\mathtt{z}(\Gamma) = \mathtt{z}(\sigma(\Gamma))\).
\end{proof}
	\begin{corollary}
		Let \(\Gamma\) be a column standard Young tableau of shape $t\omega_r$ which satisfies \ref{property:P1}. Then there is a semistandard tableau \(\mathtt{s}(\Gamma)\) of shape \(t\omega_r\) such that on the Richardson variety \(X^{v_{\underline{l}}}_{w}\) we have
		\[
		f_{\Gamma} = f_{\mathtt{s}(\Gamma)}.
		\]
		Also, \(\mathtt{s}(\Gamma)\) satisfies \ref{property:P1} and \(\mathtt{z}(\Gamma) = \mathtt{z}(\mathtt{s}(\Gamma))\).
		\label{cor:straightenstep12}
	\end{corollary}
	\begin{proof}
		Proof is by induction on row indices. We write row \( 1 \) is in non-decreasing order by rearranging columns of \(\Gamma\). Let \( 1 < i \leq r \) and assume that all rows from \( 1 \) to \( i-1 \) are in non-decreasing order. We convert row \(i\) in non-decreasing order by applying following step multiple times.
		\begin{center}\fbox{
		\begin{minipage}{0.9\textwidth}
			Let \(\Gamma(i,j) > \Gamma(i,j+1)\) for some $1 \leq j \leq t-1$. Then by using \cref{lemma:firstpartrel} on columns \(j,j+1\) we swap the entries in boxes \((i,j),(i,j+1)\) without changing any other entry in \(\Gamma\).
		\end{minipage}}
		\end{center}	
		Hence we get a semistandard tableau \( \mathtt{s}(\Gamma) \) such that \( f_{\Gamma} = f_{\mathtt{s}(\Gamma)} \) and  \(\mathtt{s}(\Gamma)(i,j) \in \{l_{i-1}, \ldots, a_{i-1}+1\}\). Hence $\mathtt{s}(\Gamma)$ satisfies \ref{property:P1}.	
	\end{proof}
	\begin{lemma}
		Let \(\Gamma\) be a column standard Young tableau of shape $2\omega_r$ which satisfies \ref{property:P2} such that \( col_1(\Gamma)=\alpha = (\alpha_1, \ldots, \alpha_r),\) and \(col_2(\Gamma)=\beta = (\beta_1, \ldots, \beta_r) \in I(r,n)\). Let \(1 \le k \le r\) be such that \(\alpha_r \le \beta_r,\cdots,\alpha_{k+1} \le \beta_{k+1}, \alpha_{k} > \beta_{k}\). Then on the Richardson variety \(X^{v_{\underline{l}}}_{w}\), the Pl\"ucker relation $\mathcal{P}(\alpha, \beta, k)$ is 
		\[
		p_{\alpha}p_{\beta} - p_{\alpha^{\sigma}}p_{\beta^{\sigma}}=0
		\]
		where \( \sigma \) is the transposition  \((\alpha_{k},\beta_{k}) \). 
		
		Moreover, let $\sigma(\Gamma)$ be the column standard Young tableau of shape $2\omega_r$ with \(col_1(\sigma(\Gamma))=\alpha^{\sigma}\) and \(col_2(\sigma(\Gamma))=\beta^{\sigma}\). Then $f_{\Gamma}=f_{\sigma(\Gamma)}$. Also, \(\sigma(\Gamma)\) satisfies \ref{property:P2}.
		\label{lemma:secondpartrel}
	\end{lemma}
	\begin{proof}
		We will prove that in the Pl\"ucker relation \(\mathcal{P}(\alpha,\beta,k) \), the monomial \( p_{\alpha^{\sigma}}p_{\beta^{\sigma}}\) is non-zero on the Richardson variety $X^{v_{\underline{l}}}_w$ if and only if \( \sigma =id \) or \(\sigma = (\alpha_{k},\beta_{k})\). Observe that if $\sigma$ is other than $id$ and the transposition $(\alpha_k, \beta_k)$ then either (a) \(\alpha_j\in \beta^{\sigma}\) for some \(k+1 \le j \le r\) or (b) \(\alpha_{k},\beta_{k} \in \beta^{\sigma}\).
		
		(a) Let \(\sigma\) be such that \(\alpha_j\in \beta^{\sigma}\) for some \(k+1 \le j \le r\). We want to understand the position of \(\alpha_j\) in \(\beta^{\sigma}\). Let \(1 \le i \le r\) be such that \(\beta^{\sigma}_i = \alpha_j\). Note that \( \{\beta_{r}, \ldots, \beta_{j}\} \) all are in \( \beta^{\sigma} \) and bigger than  \( \alpha_j\), hence \(i<j\). Since \(\beta^{\sigma}_i \in \{a_{j-1}+1, \ldots, a_j\}\) and $i < j$ implies \(\{a_{j-1}+1, \ldots, a_j\} \cap \{l_{i-1}, \ldots, a_i\} = \emptyset\), we have that the Pl\"ucker coordinate \(p_{\beta^{\sigma}}\) is zero on \(X^{v_{\underline{l}}}_w\). 
		
		(b) Let \(\sigma\) be such that \(\alpha_{k},\beta_{k} \in \beta^{\sigma}\). We want to understand the position of $\beta_k$ in $\beta^{\sigma}$. Let $1 \leq i \leq r$ be such that \(\beta^{\sigma}_i = \beta_{k}\). Observe that \(\beta^{\sigma}\) contains \(\{\alpha_k,\beta_{k+1}, \ldots, \beta_{r}\}\) and all are bigger than \(\beta_k\). Hence we have \(i < k\). Since \(\beta^{\sigma}_i \in \{a_{k-1}+1, \ldots, a_{k}\}\) and $i<k$ implies \(\{a_{k-1}+1, \ldots, a_{k}\} \cap \{l_{i-1}, \ldots, a_i\} = \emptyset\), we have the Pl\"ucker coordinate \(p_{\beta^{\sigma}}\) is zero on \(X^{v_{\underline{l}}}_w\).
		
	
	Note that $\alpha^{\sigma}=(\alpha_1 < \cdots < \alpha_{k-1} < \beta_k< \alpha_{k+1}< \cdots< \alpha_r)$ follows from $\alpha_i, \beta_i \in \{a_{i-1}+1, \ldots, a_{i}\}$ (see \ref{property:P2}). Similarly, we have $\beta^{\sigma}=(\beta_1< \cdots< \beta_{k-1} < \alpha_k < \beta_{k+1}< \cdots< \beta_r)$. It follows that \(\sigma(\Gamma)\) satisfies \ref{property:P2}. 
	
\end{proof}

	\begin{corollary}
		\label{cor:straightenstep22}
		Let \(\Gamma\) be a column standard Young tableau of shape $t\omega_r$ which satisfies \ref{property:P2}. Then there is a semistandard tableau \(\mathtt{s}(\Gamma)\) of shape \(t\omega_r\) which is non-zero on the Richardson variety \(X^{v_{\underline{l}}}_{w}\) such that
		\[
		f_{\Gamma} = f_{\mathtt{s}(\Gamma)}
		\]
		Also, we have \(\mathtt{s}(\Gamma)\) satisfies \ref{property:P2}.
	\end{corollary}
	\begin{proof}
		Proof is by descending induction.  We write row \( r \) in non-decreasing order by rearranging columns of \(\Gamma\). Let \( 1 \le i < r \) and assume that all rows from \( i+1 \) to \( r \) are in non-decreasing order. 
		\begin{center}\fbox{
				\begin{minipage}{0.9\textwidth}
				Let \(\Gamma(i,j) > \Gamma(i,j+1)\) for some $1 \leq j \leq t-1$. Then by applying \cref{lemma:secondpartrel} on columns \(j,j+1\), we swap the entries in boxes \((i,j),(i,j+1)\) without changing any other entry in \(\Gamma\).
			\end{minipage}}
		\end{center}
	 After multiple applications of above step we write row \( i \) in non-decreasing order. Hence we get row semistandard tableau \( \mathtt{s}(\Gamma) \) such that \( f_{\Gamma} = f_{\mathtt{s}(\Gamma)} \) and  \(\mathtt{s}(\Gamma)(i,j) \in \{a_{i-1}+1, \ldots, a_i\}\). Hence $\mathtt{s}(\Gamma)$ satisfies \ref{property:P2}.
	\end{proof}
	\begin{theorem}
		Let \(\Gamma_1\) and \(\Gamma_2\) be two \(T\)-invariant semistandard Young tableaux of shape \(nd_1\omega_r\) and \(nd_2\omega_r\) respectively on $X^{v_{\underline{l}}}_w$. Then there exists a semistandard \(T\)-invariant Young tableau \(\mathtt{s}(\Gamma_1\Gamma_2)\) such that \(f_{\Gamma_1}f_{\Gamma_2} = f_{\mathtt{s}(\Gamma_1\Gamma_2)} \) with the property that \(\mathtt{z}(\mathtt{s}(\Gamma_1\Gamma_2)) = \mathtt{z}(\Gamma_1)+\mathtt{z}(\Gamma_2)\).
		\label{thm:straightenTinv}
	\end{theorem}
	\begin{proof}
		Using \cref{def:split} we have \(\Gamma_1 = \Gamma_1^{(1)}\Gamma_1^{(2)}\) and \(\Gamma_2 = \Gamma_2^{(1)}\Gamma_2^{(2)}\). 
		By \ref{def:split:P1P2z} we have \(\Gamma_1^{(1)}\) and \(\Gamma_2^{(1)}\) satisfy ~\ref{property:P1}. Hence by \labelcref{prop-remark:mult} we have \(\Gamma_1^{(1)}\Gamma_2^{(1)}\) satisfies ~\ref{property:P1}. Also by \cref{def:z:adds}  we have
		\begin{align*}
			\mathtt{z}(\Gamma_1^{(1)}\Gamma_2^{(1)}) 
			= \mathtt{z}(\Gamma_1^{(1)})+\mathtt{z}(\Gamma_2^{(1)}) 
		\end{align*} By \cref{cor:straightenstep12} there is semistandard tableau 	\(\mathtt{s}(\Gamma_1^{(1)}\Gamma_2^{(1)})\) such that \(f_{\Gamma_1^{(1)}\Gamma_2^{(1)}} = f_{\mathtt{s}(\Gamma_1^{(1)}\Gamma_2^{(1)})}\) with the properties that \(\mathtt{s}(\Gamma_1^{(1)}\Gamma_2^{(1)})\) satisfy \ref{property:P1} and
		\begin{align}
			\mathtt{z}(\mathtt{s}(\Gamma_1^{(1)}\Gamma_2^{(1)})) 
			&= \mathtt{z}(\Gamma_1^{(1)}\Gamma_2^{(1)}) \nonumber\\
			&=\mathtt{z}(\Gamma_1^{(1)})+\mathtt{z}(\Gamma_2^{(1)}) &&\text{(by \cref{def:z:adds})} \nonumber\\
			&=\mathtt{z}(\Gamma_1)+\mathtt{z}(\Gamma_2) &&\text{(by \cref{def:split:P1P2z})}.
				\label{eq:firstpartz}
		\end{align}
	
	Similarly By \cref{def:split:P1P2z} we have \(\Gamma_1^{(2)}\) and \(\Gamma_2^{(2)}\) satisfy ~\ref{property:P2}. Using \labelcref{prop-remark:mult} we have \(\Gamma_1^{(2)}\Gamma_2^{(2)}\) satisfy ~\ref{property:P2}.
		By \cref{cor:straightenstep22} there is semistandard tableau \(\mathtt{s}(\Gamma_1^{(2)}\Gamma_2^{(2)})\) such that \(f_{\Gamma_1^{(2)}\Gamma_2^{(2)}} = f_{\mathtt{s}(\Gamma_1^{(2)}\Gamma_2^{(2)})}\) with the properties that \(\mathtt{s}(\Gamma_1^{(2)}\Gamma_2^{(2)})\) satisfy \ref{property:P2}. Hence \begin{equation}
		\mathtt{z}(\mathtt{s}(\Gamma_1^{(2)}\Gamma_2^{(2)}))=0.
		\label{eq:secondpartz0}
		\end{equation}
			We define 
		\[\mathtt{s}(\Gamma_1\Gamma_2) := \mathtt{s}(\Gamma_1^{(1)}\Gamma_2^{(1)})\mathtt{s}(\Gamma_1^{(2)}\Gamma_2^{(2)})\]
		\(\mathtt{s}(\Gamma_1\Gamma_2)\) is semistandard follows from \labelcref{prop-remark:sst}. We have
		\begin{align*}
			f_{\Gamma_1}f_{\Gamma_2} 
			&= f_{\Gamma_1^{(1)}\Gamma_1^{(2)}}f_{\Gamma_2^{(1)}\Gamma_2^{(2)}}\\
			&= f_{\Gamma_1^{(1)}}f_{\Gamma_2^{(1)}}f_{\Gamma_1^{(2)}}f_{\Gamma_2^{(2)}}\\
			&= f_{\Gamma_1^{(1)}\Gamma_2^{(1)}}f_{\Gamma_1^{(2)}\Gamma_2^{(2)}}\\
			&= f_{\mathtt{s}(\Gamma_1^{(1)}\Gamma_2^{(1)})}f_{\mathtt{s}(\Gamma_1^{(2)}\Gamma_2^{(2)})}\\
			&= f_{\mathtt{s}(\Gamma_1^{(1)}\Gamma_2^{(1)})\mathtt{s}(\Gamma_1^{(2)}\Gamma_2^{(2)})}\\
			&= f_{\mathtt{s}(\Gamma_1\Gamma_2)}
		\end{align*}
		We now prove \(\mathtt{z}(\mathtt{s}(\Gamma_1\Gamma_2)) = \mathtt{z}(\Gamma_1)+\mathtt{z}(\Gamma_2)\).
		\begin{align*}
			\mathtt{z}(\mathtt{s}(\Gamma_1\Gamma_2)) 
			&=\mathtt{z}(\mathtt{s}(\Gamma_1^{(1)}\Gamma_2^{(1)})\mathtt{s}(\Gamma_1^{(2)}\Gamma_2^{(2)}))\\
			&=\mathtt{z}(\mathtt{s}(\Gamma_1^{(1)}\Gamma_2^{(1)}))
						+\mathtt{z}(\mathtt{s}(\Gamma_1^{(2)}\Gamma_2^{(2)})) &&\text{(by \cref{def:z:adds})}\\
			&= \mathtt{z}(\Gamma_1)+\mathtt{z}(\Gamma_2) &&\text{(by~\ref{eq:firstpartz} and \ref{eq:secondpartz0})}.
		\end{align*}
	\end{proof}

\subsection{Main theorem}\label{sec:proofmain}
	In this section we give a proof of \cref{thm:main}. First we construct the homogeneous coordinate ring of ${\mathbb P}^{a_1-l_1} \times {\mathbb P}^{a_2-l_2} \times \cdots \times {\mathbb P}^{a_{r-1}-l_{r-1}}$ embedded via the very ample line bundle $\mathcal{O}_{\mathbb{P}^{a_1-l_1}}(c_1) \boxtimes \mathcal{O}_{\mathbb{P}^{a_2-l_2}}(c_2) \boxtimes \cdots \boxtimes \mathcal{O}_{\mathbb{P}^{a_{r-1}-l_{r-1}}}(c_{r-1})$ and then we prove it is isomorphic to the algebra \(R\).
	\begin{definition}
		For \(1 \leq i \leq r-1\), let $\mathbb{C}[x_j : j \in C_{i,2}]$ be the polynomial ring generated by the indeterminate $x_j$'s. For \(d\in \mathbb{Z}_{\ge0}\), let \(A^{(i)}_d = \mathbb{C}[x_j : j \in C_{i,2}]_{dc_i}\) be the \(dc_{i}\)-th graded component of \(\mathbb{C}[x_{j}:j \in C_{i,2}]\). Let \(A_d =  \Big(\bigotimes_{i=1}^{r-1} A^{(i)}_d\Big)\).  A vector space basis of \(A_d\) is
		\[\mathfrak{A}_d=\Big\{
			\bigotimes_{i=1}^{r-1} \prod_{j \in C_{i,2}} x_j^{e_j} 
					: e_j \in \mathbb{Z}_{\ge 0}, \sum_{j \in C_{i,2}} e_j = dc_i
		\Big\}\]
		Define graded algebras \(A^{(i)} = \bigoplus_{d\ge0} A^{(i)}_d\) and \(A = \bigoplus_{d \ge 0} A_d\) with component wise multiplication.
		\label{def:segrering}
	\end{definition} 
	\begin{lemma}
		For \(d \in \mathbb{Z}_{\ge 0}\), the map \(\bar{\mathtt{z}}_d\) induces an isomorphism of vector spaces from $R_d$ to $A_d$. 
		\label{cor:vecspiso-R-A}
	\end{lemma}
	\begin{proof}
		Note that the natural map \[\Phi_d:  \mathtt{P}_d \longrightarrow \mathfrak{A}_d\] given by 
		\[ \bar{z} \mapsto \bigotimes_{i=1}^{r-1}\prod_{j \in C_{i,2}} x_j^{z_j} \] is a bijection.
		Then \[\begin{tikzcd}
		\mathtt{ST}(\lambda_d) \arrow[r, "\bar{\mathtt{z}}_d"] & \mathtt{P}_d   \arrow[r, "\Phi_d "] & \mathfrak{A}_d
		\end{tikzcd}\]
		given by
		\[
		f_{\Gamma} \mapsto  \bar{\mathtt{z}}_d(\Gamma) \mapsto  \bigotimes\limits_{i=1}^{r-1}\prod\limits_{j \in C_{i,2}} x_j^{\mathtt{z}_j(\Gamma)}
		\]
%
	is a bijection.	Thus \(\Phi_d \circ \bar{\mathtt{z}}_d\) induces an isomorphism of vector spaces from $R_d$ to $A_d$. 
	\end{proof}
\begin{definition}
	We define an isomorphism of vector spaces $\Phi: R \longrightarrow A$ as follows. For $d \in \mathbb{Z}_{\geq 0}$ we define $\Phi|_{R_d}=\Phi_d \circ \bar{\mathtt{z}}_d $ where $\Phi_d$ is as in \cref{cor:vecspiso-R-A}.
\end{definition}
	\begin{lemma} 
 $\Phi: R \longrightarrow A$ is an isomorphism of graded $\mathbb{C}$-algebras.
	\end{lemma}
	\begin{proof}
		It is enough to prove that $\Phi$ is a homomorphism of rings.
		
		Let \(\Gamma_1,\Gamma_2 \in \mathtt{ST}\). Using \cref{thm:straightenTinv} we have \(f_{\Gamma_1}f_{\Gamma_2} = f_{\mathtt{s}(\Gamma_1\Gamma_2)}\) such that $\mathtt{s}(\Gamma_1\Gamma_2)$ is semistandard. Now we prove \(\Phi\) is a ring homomorphism.
		\begin{align*}
			\Phi(f_{\Gamma_1}f_{\Gamma_2}) &= \Phi(f_{\mathtt{s}(\Gamma_1\Gamma_2)}) \\
			&= \bigotimes_{i=1}^{r-1} \prod_{j \in C_{i,2}} x_j^{\mathtt{z}_j({\mathtt{s}(\Gamma_1\Gamma_2)})} \\
			&= \bigotimes_{i=1}^{r-1} \prod_{j \in C_{i,2}} x_j^{\mathtt{z}_j(\Gamma_1) + \mathtt{z}_j(\Gamma_2)} && \text{(by \cref{thm:straightenTinv})}\\
			&= \Big(\bigotimes_{i=1}^{r-1} \prod_{j \in C_{i,2}} x_j^{\mathtt{z}_j(\Gamma_1)}\Big) 
			\Big(\bigotimes_{i=1}^{r-1} \prod_{j \in C_{i,2}} x_j^{\mathtt{z}_j(\Gamma_2)}\Big)\\
			&= \Phi(f_{\Gamma_1})\Phi(f_{\Gamma_2}) 
		\end{align*}
		Thus $\Phi: R \longrightarrow A$ is an isomorphism of graded $\mathbb{C}$-algebras.
	\end{proof}
\noindent	We now give the proof of main theorem.
	\begin{proof}[Proof of \cref{thm:main}]
	The algebra \(A^{(i)}\) is the homogeneous coordinate ring of \(c_i\)-th Veronese embedding of \(\mathbb{P}^{|C_{i,2}|-1}\) and the algebra \(A\) is the homogeneous coordinate ring of Segre embedding of the product of these spaces \cite[Exercise 13.14, p.299]{Eisenbud1995}. Let \(Y = \mathbb{P}^{|C_{1,2}|-1} \times \cdots \times \mathbb{P}^{|C_{r-1,2}|-1}\) and \({\cal L} = \mathcal{O}(c_1) \boxtimes \cdots \boxtimes \mathcal{O}(c_{r-1})\). Then we have \(\oplus_{d\ge 0} H^0(Y,{\cal L}^{\otimes d}) = A \simeq \oplus_{d\ge 0} R_d\).
\end{proof}

\appendix	
\section{Proof of \cref{lemma:Tinv-eq-condition}}\label{proof:Tinv_eq_condition}	

\begin{proof}
	\((\implies)\) For (1), $B_i$ is skew shape follows from definition of $B_i$ and subtableau of \(\Gamma\) containing boxes \(B_{i}\) is row semistandard follows from $\Gamma$ is semistandard. (2) follows from \cref{cor:BicontainsCi} and (3) follows from $\Gamma$ is non-zero on $X^{v_{\underline{l}}}_w$. 
	
	\((\impliedby)\) From (2) we have \(\Gamma\) is \(T\)-invariant. Now, it is sufficient to prove that $\Gamma$ is a non-zero semistandard tableau on \(X^{v_{\underline{l}}}_w\). We do this in the following four steps.
	
	\begin{enumerate}[label=\text{Step} \arabic*:,ref=\text{Step} \arabic*, wide]
		\item\label{bij:rowstd}\ \emph{\( \Gamma\) is row semistandard.} First row is semistandard follows from filling of \(B_1\) is row semistandard (given (1)). Let \(2 \le i \le r\). Consider two boxes \((i,j),(i,j')\) in row \(i\) such that \(j < j'\) and we prove that \(\Gamma(i,j)\le\Gamma(i,j')\). Since the boxes in row $i$ are \(B_{i-1,2}\sqcup B_{i,1}\)  we have following three cases
		\begin{multicols}{2}
			\begin{enumerate}[label=(\alph*),ref=(\alph*),itemsep=0pt, topsep=0pt]
				\item \( (i,j),(i,j')\in B_{i-1,2}\), 
				\item \( (i,j),(i,j')\in B_{i,1}\),
				\item \( (i,j)\in B_{i-1,2}\) and \((i,j')\in B_{i,1}\). 
			\end{enumerate}
		\end{multicols}
		For (a) and (b), \(\Gamma(i,j) \le \Gamma(i,j')\) follows from \(B_{i-1}\) and \(B_i\) are row semistandard (given (1)). For (c), using given (2) we have \(B_{i-1}\) and \(B_i\) contain entries from \(C_{i-1} = \{a_{i-2}+1, \ldots, a_{i-1}\}\) and \(C_i = \{a_{i-1}+1, \ldots, a_i\}\) respectively, hence we have \(\Gamma(i,j) \le \Gamma(i,j')\).
		
		\item \( \Gamma(i,nd) = a_i \). Observe that the box \((i,nd) \in B_{i,1}\). Hence from given (2) we have \(\Gamma(i,nd) \in C_i\). Since \(a_i\) appears with multiplicity \(rd\) in \(B_i\) (given (2)) and \(|B_{i,2}| = dc_i <dr\) (\cref{lemma:property_c}), \(a_i\) appears with non-zero multiplicity in \(B_{i,1}\). Further since \(B_{i}\) is row semistandard (given(1)) and \( a_i \) is largest integer in \(C_i\) we have \( \Gamma(i,nd) = a_i \).
		
		\item\label{step:colrd}\(\Gamma(i,rd) = a_{i-1}+1\).
		For \(i=1\) the box \((i,rd) \in B_{i,1}\). For \(2 \le i \le r\), the box \((i,rd) \in B_{i-1,2} \sqcup B_{i,1}\). Since \(c_{i-1} < r\) (\cref{lemma:property_c}), we have box \((i,rd) \in B_{i,1}\). 
		
		By given (2) we have $a_{i-1}+1 (\in C_{i})$ is filled in boxes $B_i (= B_{i,1}\sqcup B_{i,2})$ with multiplicity \(rd\). Using given (3) we have $\Gamma(i+1,1) \geq l_i \geq a_{i-1}+2$ and using (1) we have $B_{i,2}$ is row semistandard, it follows that $a_{i-1}+1$ can not appear in $B_{i,2}$. Thus $a_{i-1}+1$ appears in $B_{i,1}$ with multiplicity $rd$. Since \(B_{i,1}\) is row semistandard (given (1)) and \(a_{i-1}+1\) is smallest integer in \(C_{i}\) we have \(\Gamma(i,j) = a_{i-1}+1\) for all $dc_{i-1}+1 \leq j \leq dc_{i-1}+rd$. Since by \cref{lemma:property_c} $dc_i < dr$, we have $\Gamma(i,rd)=a_{i-1}+1$.
		
		\item \emph{\( \Gamma\) is column standard.} We have to prove \( \Gamma(i,j)<\Gamma(i+1,j) \) for \( 1 \le i \le r-1\) and \(1 \le j \le nd \). If \( j \le rd \) then by Step $1$, Step $3$ and given (3), we have \( \Gamma(i,j) \in [l_{i-1},a_{i-1}+1]\). Hence \( l_{i}\ge a_{i-1}+2 \) gives \( \Gamma(i,j) < \Gamma(i+1,j) \).
		If \( j > rd \) then by Step $2$ and Step $3$ we have \( \Gamma(i,j) \in [a_{i-1}+1,a_i] \). Hence \( \Gamma(i,j) < \Gamma(i+1,j) \).				
	\end{enumerate}	
	
	Therefore $\Gamma$ is semistandard follows from Step $1$ and Step $4$ and $\Gamma$ is non-zero on $X^{v_{\underline{l}}}_w$ follows from Step $2$ and given \(\Gamma(i,1) \ge l_{i-1}\).	
\end{proof}
	\bibliography{Tquot-ric-var}
	\bibliographystyle{alpha}
\end{document}